\documentclass[12pt]{amsart}
\usepackage{amssymb,amsmath,amsthm}
\usepackage{a4}
\usepackage{color}
\usepackage[colorlinks=true,
linkcolor=webgreen,
filecolor=webbrown,
citecolor=webgreen]{hyperref}
\definecolor{webgreen}{rgb}{0,0,1}
\definecolor{recrown}{rgb}{1,.2,.6}
\usepackage{fullpage}

\begin{document}
\newtheorem{theorem}{Theorem}
\newtheorem{corollary}[theorem]{Corollary}
\newtheorem{lemma}[theorem]{Lemma}
\theoremstyle{example}
\newtheorem*{example}{Examples}
\newtheorem{conjecture}[theorem]{Conjecture}
\newtheorem{thmx}{\bf Theorem}
\renewcommand{\thethmx}{\text{\Alph{thmx}}}
\newtheorem{lemmax}{Lemma}
\renewcommand{\thelemmax}{\text{\Alph{lemmax}}}
\hoffset=-0cm
\theoremstyle{definition}
\newtheorem*{definition}{Definition}
\theoremstyle{remark}
\newtheorem*{remark}{\bf Remark}
\title{\bf A New Class of Irreducible Polynomials}
\author{Jitender Singh$^\dagger$}
\address{$~^\dagger$ Department of Mathematics, Guru Nanak Dev University, Amritsar-143005, India}
\author{Sanjeev Kumar$^{\ddagger,*}$}
\address{$~^\ddagger$ Department of Mathematics, SGGS College, Sector-26, Chandigarh-160019, India}
\date{}
\maketitle
\parindent=0cm
\footnotetext[2]{sonumaths@gmail.com}
\footnotetext[3]{$^{,*}$Corresponding author: sanjeev\_kumar\_19@yahoo.co.in}
\begin{abstract}
In this article, we propose a few sufficient conditions on polynomials having integer coefficients all of whose zeros lie outside a closed disc centered at the origin in the complex plane and deduce the irreducibility over the ring of integers. 
\end{abstract}

\section{Introduction}
Testing polynomials for irreducibility over a given domain is an arduous task. Of particular interest are the polynomials having integer coefficients for which some well--known classical irreducibility criteria due to Sch\"{o}nemann, Eisenstein, and Dumas exist (see \cite{S,E,D} and for an insightful historical account of Sch\"{o}nemann and Eisenstein criteria, see \cite{Cox}). Recently, the elegant criteria established in \cite{Mu,G} turn out to be extremely significant keeping in view their intimate connection with prime numbers. Moreover, the notion of locating the zeros of the given polynomial being tested for irreducibility is quite informative (see \cite{P}). In this regard, one can infer that if for each zero $\zeta$ of  $g\in \mathbb Z[x]$,  $|\zeta|\leq r$ holds for some  $r>0$, then each zero $\theta$ of $f=g(x-c)$ is given  by $\theta=\zeta+c$ which on applying the triangle inequality yields  $|\theta|>1$ for  any integer  $c$ whose absolute value exceeds $r+2$. Also, the translational invariance of irreducibility of polynomials in the ring $\Bbb{Z}[x]$ ensures the irreducibility of $g$ vis--\'a--vis from that of $f$. Proceeding in this manner, one can frame the following irreducibility criterion from that of the one given in \cite[Theorem 1]{G}.
\begin{thmx}\label{th:0}
 Let  $f\in \mathbb{Z}[x]$ be such that each zero $\theta$ of $f$ satisfies $|\theta|>d$. If $f(0)=\pm pd$ for some positive integer $d$ and prime $p\nmid d$, then $f$ is irreducible in $\mathbb{Z}[x]$.
\end{thmx}
\begin{proof}
   If possible, let $f(x)=f_1(x)f_2(x)$, where $f_1$ and $f_2$ are non--constant polynomials in $\mathbb{Z}[x]$. By hypothesis on $f$, $f(0)=f_1(0)f_2(0)=\pm pd$ which shows that $p$ divides exactly one of the factors $f_1(0)$ or $f_2(0)$. Assume without loss of generality that $p\mid f_2(0)$. Then $|f_1(0)|\leq d$. On the other hand if $c\neq 0$ is the leading coefficient of $f_1$, then we may write
   \begin{equation}\label{c1}
    f_1(x)=c\prod_{\theta}(x-\theta),
   \end{equation}
    where the product runs over all zeros of $f_1$. By the hypothesis on zeros of $f$ we must have from \eqref{c1} that $|f_1(0)|=|c|\prod_\theta |\theta|>|c| d^{\deg{f_1}}\geq d$, a contradiction.
\end{proof}

In Theorem \ref{th:0}, the primality of $|f(0)|/d$ is necessary to deduce the irreducibility. In an attempt to weaken the hypothesis, we confront the following natural question:
 \emph{Given $|\theta|>d$ for each zero $\theta$ of $f$, is it still possible to recover the irreducibility of $f$ if instead $|f(0)|/d$ is a prime power$?$}
Nevertheless, under certain mild conditions on the coefficients of $f$, we show that the answer to the above question is in the affirmative.

Recall that a polynomial $f$ having integer coefficients is primitive if the greatest common divisor of all its coefficients is 1.
Our main results are the following:
\begin{theorem}\label{th:1}
   Let $f=a_0+ a_{1}x+\cdots+a_n x^n\in \Bbb{Z}[x]$ be a primitive polynomial such that each zero $\theta$ of $f$ satisfies $|\theta|>d$, where $a_0=\pm p^k d$ for some positive integers $k$ and $d$, and a prime $p\nmid d$. If $j\in\{1,\ldots,n\}$ is such that $\gcd(k,j)=1$,  $p^k\mid \gcd(a_0,a_1,\ldots,a_{j-1})$ and for $k>1$,  $p\nmid a_{j}$, then $f$ is irreducible in $\Bbb{Z}[x]$.
\end{theorem}
\begin{theorem}\label{th:2}
    Let $f=a_0+ a_{1}x+\cdots+a_n x^n\in \Bbb{Z}[x]$ be a primitive polynomial such that each zero $\theta$ of $f$ satisfies $|\theta|>d$, where $a_n=\pm p^k d$ for some positive integer $k$ and  $d$, and a prime $p\nmid d$. Let $j\in \{1,\ldots, n\}$ be such that $\gcd(k,j)=1$, $p^k\mid \gcd(a_{n-j+1},a_{n-j+2},\ldots,a_{n})$ and for $k>1$,  $p\nmid a_{n-j}$. If $|a_0/q|\leq |a_n|$ where $q$ is the smallest prime divisor of $a_0$, then $f$ is irreducible in $\Bbb{Z}[x]$.
\end{theorem}
To prove Theorems \ref{th:1}-\ref{th:2}, elementary divisibility theory for integers is devised. The cogent techniques involved in the proofs are of independent interest as well. Further, the notations specified below are imperative and shall be used in the sequel.

\noindent \textbf{Notations.}  If $f(x)=f_1(x)f_2(x)$, unless otherwise specified, we write $f=a_0+ a_{1}x+\cdots+a_n x^n\in \Bbb{Z}[x]$; $f_1=b_0+b_1x+\cdots+b_mx^m$ and  $f_2=c_0+c_1x+\cdots+c_{n-m}x^{n-m}$ are non--constant polynomials in $\Bbb{Z}[x]$.
 Define further that
 \begin{equation*}
    b_{m+1}=b_{m+2}=\cdots=b_n=0;~c_{n-m+1}=c_{n-m+2}=\cdots=c_{n}=0,
 \end{equation*}
 so that we may write
\begin{equation}\label{ee}
 a_t=b_0 c_{t}+b_1 c_{t-1}+\cdots+b_tc_0,~\text{for each}~t=0,1,\ldots,n.
\end{equation}
\section{Proofs of Theorems \ref{th:1}-\ref{th:2}}
To prove Theorems \ref{th:1}-\ref{th:2}, we first prove the following crucial result.
\begin{lemma}\label{L1}
    Let $f=a_0+ a_{1}x+\cdots+a_n x^n$, $f_1=b_0+b_1x+\cdots+b_mx^m$, and  $f_2=c_0+c_1x+\cdots+c_{n-m}x^{n-m}$ be non--constant polynomials in $\Bbb{Z}[x]$ such that $f(x)=f_1(x)f_2(x)$. Suppose that there is a prime number $p$ and positive integers $k\geq 2$ and $j\leq n$ such that $p^k\mid \gcd(a_0,a_1,\ldots,a_{j-1})$, $p^{k+1}\nmid a_0$, and $\gcd(k,j)=1$. If $p\mid b_0$ and $p\mid c_0$, then $p\mid a_j$.
 \end{lemma}
\textbf{Proof of Lemma \ref{L1}.} In view of the hypothesis that $p\mid b_0$ and $p\mid c_0$, there exists a positive integer $\ell\leq k-\ell$ such that $p^\ell\mid b_0$ and $p^{k-\ell}\mid c_0$, where $\ell$ and $k-\ell$ are highest powers of $p$ dividing $b_0$ and $c_0$ respectively.
  To proceed we define the nonnegative integer $\kappa$ such that $\kappa=(j-2)/2$ if $j$ is even and $\kappa=(j-1)/2$ if $j$ is odd.
   We now arrive at the following cases:

   \textbf{Case I:} $\ell<k-\ell$. In this case we have the following subcases:

   \textbf{Subcase I:} $p\mid b_i$ for all $i=0,\ldots, \kappa$.
   Using the expressions for $a_i$ and $a_{2i}$ successively for each $i=0, \ldots, \kappa$, we find that $p$ divides $c_0$, $c_1$, $\ldots$, $c_{\kappa}$. If $\alpha_i$ and $\beta_i$ are the highest powers of $p$ dividing $b_i$ and $c_i$ respectively, then  $\alpha_0=\ell$ and $\beta_0=k-\ell$. We claim that $\alpha_i\geq \ell$ and $\beta_i\geq k-\ell$ for all $i\leq \kappa$. For proof, we consider $a_1=b_0c_1+b_1c_0$ which tells us that  $$\ell+\beta_1\geq k,~\beta_1\leq  k-2\ell+\alpha_1,$$ which further give $\alpha_1\geq\ell$ and $\beta_1\geq k-\ell$ with $\alpha_1<\beta_1$ since $\ell<k-\ell$.  Then $p^k\mid (a_2-b_1c_1)=b_0c_2+b_2c_0$ which for the similar reasons shows that $\alpha_2\geq \ell$ and $\beta_2\geq k-\ell$ with $\alpha_2<\beta_2$.
   Continuing in this manner, suppose for some positive integer $i^*<\kappa$ that the following have been proved successively
   \begin{equation}\label{a6}
    \alpha_{i}\geq \ell,~\beta_{i}\geq k-\ell,~\alpha_i<\beta_i,~\text{for each}~i=0,1,\ldots,i^*.
   \end{equation}
Then consider $a_{i^*+1}=b_0c_{i^*+1}+(b_1c_{i^*}+\cdots+b_{i^*}c_1)+b_{i^*+1}c_0$, where from \eqref{a6} we get $p^{\ell}\mid b_i$ and $p^{k-\ell}\mid c_{i^*+1-i}$  for each $i=1,\ldots, i^*$ so that $p^k\mid b_i c_{i^*+1-i}$. Consequently, $p^k\mid (b_1c_{i^*}+\cdots+b_{i^*}c_1)$. Also, by the hypothesis, $p^k\mid a_{i^*+1}$.  So we get $p^{k}\mid (a_{i^*+1}-b_1c_{i^*}-\cdots-b_{i^*}c_1)=b_0c_{i^*+1}+b_{i^*+1}c_0$. This proves that $\alpha_{i^*+1}\geq \ell$ and $\beta_{i^*+1}\geq k-\ell$ with $\alpha_{i^*+1}<\beta_{i^*+1}$ since $\ell<k-\ell$. With this, we conclude that
\begin{equation}\label{e1}
 \alpha_i\geq \ell,~\beta_i\geq k-\ell,~\alpha_i<\beta_i~\text{for all}~i=0,\ldots, \kappa.
\end{equation}
To proceed further, we first assume that $\kappa=(j-2)/2$. Using \eqref{e1} in the expression for $a_{j-1}$ in \eqref{ee}, we have
   \begin{equation*}
   \begin{split}
    p^k\mid &(a_{j-1}-b_0c_{j-1}-\cdots-b_{(j-4)/2}c_{(j+2)/2}-b_{(j+2)/2}c_{(j-4)/2}- \cdots-b_{j-1}c_0)\\&=b_{(j-2)/2}c_{j/2}+b_{j/2}c_{(j-2)/2},
    \end{split}
   \end{equation*}
which shows that $p^{k-2\ell}\mid c_{j/2}$.    Consequently
   \begin{equation*}
   p\mid  \{b_0c_{j}+\cdots+b_{(j-2)/2}c_{(j+2)/2}+b_{j/2}c_{j/2}+b_{(j+2)/2}c_{(j-2)/2}+\cdots+b_{j-1}c_0\}=a_j,
   \end{equation*}
where the equality follows from \eqref{ee}.

For  $\kappa=(j-1)/2$  we have from \eqref{e1} and  \eqref{ee} that
\begin{equation*}
  p^\ell\mid (b_0c_{j}+b_1c_{j-1}+\cdots+b_{(j-1)/2}c_{(j+1)/2}+b_{(j+1)/2}c_{(j-1)/2}+\cdots+b_{j} c_0)=a_j.
\end{equation*}

\textbf{Subcase II:} There is a smallest positive integer $i\leq \kappa$ for which $p\nmid b_i$. From the Subcase I, $p^\ell$ divides each of $b_0$, $\ldots$, $b_{i-1}$ and $p^{k-\ell}$ divides each of $c_0$, $\ldots$, $c_{i-1}$. Let $q_j$ be the positive integer, such that $iq_j \leq j-1<(1+q_j)i$. Let $\beta_s$ denote the highest power of $p$ dividing $c_s$ for $i\leq s\leq j-1$. We will show that $\beta_{ti+r}=k-(t+1)\ell$, for each $t=1,\ldots q_j$ and $r=0,\ldots, i-1$.

To proceed, we first observe from \eqref{ee} that
\begin{equation}\label{a1}
b_0c_t =a_t-\mathcal{C}(c_0,c_1,\ldots,c_{t-1}),
\end{equation}
where $\mathcal{C}(c_0,\ldots, c_{t-1})$ is the integer combination of $c_0,\ldots,c_{t-1}$ which we define as follows:
\begin{equation}\label{a2}
 \mathcal{C}(c_0)=0;~\mathcal{C}(c_0,c_1,\ldots,c_{t-1})=b_{t}c_0+b_{t-1}c_{1}+\cdots+b_{1}c_{t-1} ~\text{for}~ t>1.
\end{equation}
Since $p^{k-\ell}\mid c_t$ for each $t=0,\ldots, i-1$, it follows  from \eqref{a2} that $p^{k-\ell}\mid \mathcal{C}(c_0,\ldots,c_{i-1})$, which in view of \eqref{a1} and the fact that $p^{k}\mid a_i$ gives  $\beta_i=k-2\ell$ since $p\nmid b_i$.  Suppose we have proved successively that $\beta_{i+r}={k-2\ell}$ for $0\leq r<i-1$. Then $p^{k-\ell}\mid (b_{i+r}c_0+\cdots+b_ic_{r})$ and $p^{k-\ell}\mid(b_{i+1}c_{r}+\cdots+b_1c_{i+r})$ so that from \eqref{a2}, we get $p^{k-\ell}\mid \mathcal{C}(c_0,\ldots,c_{i+r})$, which in view of \eqref{a1} gives $p^{k-2\ell}\mid c_{i+r+1}$ or $\beta_{i+r+1}\geq k-2\ell$.  Since $p\nmid b_i$, we must also have $\beta_{i+r+1}\leq k-2\ell$. So, $\beta_{i+r+1}=k-2\ell$. This proves the claim for $t=1$ and all $r=0,\ldots,i-1$.

Now suppose that $\beta_{ti+r}=k-(t+1)\ell$ for each $t=0,\ldots, t^*$ and $r=0,\ldots,i-1$ for some positive integer $t^*\leq q_j$.
Then we have
\begin{equation}\label{e3}
\begin{split}
 \alpha_s=\alpha_0;~\beta_{ti+s}&=k-(t+1)\ell ~~\text{for}~~ s=0,\ldots, i-1;~t=0,\ldots, t^*.
 \end{split}
\end{equation}
For convenience, we define
\begin{equation}\label{e4}
  h(s)=b_sc_{i(1+t^*)+r-s},~s=0,\ldots, i(1+t^*)+r.
\end{equation}
From \eqref{e3}--\eqref{e4}, we have for $r=0$  and each $s=0,\ldots,i-1$
\begin{equation}\label{a3}
p^{\ell+k-(1+t^*)\ell}\mid h(s);~p^{k-(1+t^*)\ell}\mid h(i+s);~p^{k-t^*\ell}\mid h(2i+s);~\ldots; p^{k-\ell}\mid h(i(1+t^*)+s),
\end{equation} Also, from \eqref{a2} and \eqref{e4} we have
\begin{eqnarray}\label{a4}
\nonumber  \mathcal{C}(c_0,\ldots,c_{i(1+t^*)+r-1})
&=&\sum_{s=1}^{i-1}h(s)+\sum_{s=i}^{2i-1}h(s)+\cdots+\sum_{s=it^*}^{i(1+t^*)-1}h(s)+\sum_{s=i(1+t^*)}^{i(1+t^*)+r}h(s)\\
\nonumber  &=& \sum_{s=1}^{i-1}h(s)+\sum_{s=0}^{i-1}\{h(i+s)+\cdots+h(it^*+s)\}+\sum_{s=0}^{r}h(i(1+t^*)+s)\\
  &=& \sum_{s=1}^{i-1}h(s)+\sum_{s'=1}^{t^*}\sum_{s=0}^{i-1}h(is'+s)+\sum_{s=0}^{r}h(i(1+t^*)+s).
\end{eqnarray}
Using \eqref{a3} in \eqref{a4} for $r=0$, we get $p^{k-(1+t^*)\ell}\mid \mathcal{C}(c_0,\ldots, c_{i(1+t^*)-1})$. Consequently, from \eqref{a1}, we have $p^{k-(1+t^*)\ell}\mid (a_{i(1+t^*)}-\mathcal{C}(c_0,\ldots, c_{i(1+t^*)-1}))=b_0c_{i(1+t^*)}$. This
further gives $p^{k-(2+t^*)\ell}\mid c_{i(1+t^*)}$. Thus,
\begin{equation}\label{a5}
 \beta_{i(1+t^*)+r}=k-(2+t^*)\ell>0
\end{equation}
holds for $r=0$. In view of \eqref{a5}, the assertion in \eqref{a3} holds for $r=1$, using which further in \eqref{a4} proves \eqref{a5} for $r=1$. Suppose then that  \eqref{a5} holds for each $r=0,\ldots,r^*$ for some positive integer $r^*<i-1$. Then in view of \eqref{a5} we have that \eqref{a3} holds for $r=r^*$. Using this further in  \eqref{a4} proves that \eqref{a5} holds for $r=r^*+1$. This proves the claim. So, $p^{k-(1+q_j)\ell}\mid c_s$, where $k>(1+q_j)\ell$ for all $s=0,\ldots,j-1$ which in view of
\eqref{ee} proves
\begin{equation*}
  p^{k-(1+q_j)\ell}\mid (b_0c_{j}+b_1c_{j-1}+\ldots+b_ic_{j-i}+\cdots+ b_jc_0)=a_j.
\end{equation*}

\textbf{Case II:} $\ell=k-\ell$. Here $k$ is even. Then  $j$ is odd since $\gcd(k,j)=1$. In this case, we use the fact that for any two integers $a$ and $b$, and prime $p$, if $p\mid (a+b)$ and $p\mid ab$, then $p\mid a$ and $p\mid b$.

In view of the above fact, we have from the expressions for $a_1$ and $a_2$ in \eqref{ee}  that $p\mid b_1$ and $p\mid c_1$. Similarly from the expressions for $a_2$ and $a_4$ in \eqref{ee}  we get $p\mid b_2$ and $p\mid c_2$. Continuing this way, having proved that $p$ divides each of the integers $b_0$, $c_0$, $b_1$, $c_1$, $\ldots$, $b_{{(j-3)}/{2}}$, $c_{{(j-3)}/{2}}$, it follows from the expressions for $a_{(j-1)/2}$ and $a_{j-1}$ in \eqref{ee}  that $p\mid b_{(j-1)/2}$ and $p\mid c_{(j-1)/2}$. So in view of \eqref{ee}, we get the following:
    \begin{equation*}
      p\mid (b_0c_{j}+\cdots+b_{{(j-1)}/{2}}c_{{(j+1)}/{2}}+b_{{(j+1)}/{2}}c_{{(j-1)}/{2}}+\cdots+b_{j} c_0)=a_j.
    \end{equation*}
 This completes the proof of Lemma \ref{L1}.   \qed
\begin{remark}
   Proof of  Lemma \ref{L1} becomes considerably short if one assumes $\gcd(k,j!)=1$. In that case, the condition $\gcd(k,j!)=1$ implies $k>j$ and $k-t\ell>0$ for each $t=1,\ldots,j$. Consequently in view of \eqref{ee}, one immediately finds  recursively that
\begin{equation}\label{t1}
  p^{k-(t-1)\ell}\mid (a_{t-1}-b_1c_{t-2}-b_2c_{t-3}-\cdots-b_{t-1}c_0)=b_0c_{t-1},~t=1,\ldots,j.
\end{equation}
So from \eqref{t1} it follows that $p\mid c_t$ for each $t=0,\ldots,j-1$ which in view of \eqref{ee} and the fact that $p\mid b_0$ yields the desired conclusion $p\mid (b_0c_j+b_1c_{j-1}+\cdots+b_jc_0)=a_j$.
\end{remark}
\textbf{Proof of Theorem \ref{th:1}.} 
If possible, assume that $f(x)=f_1(x)f_2(x)$ where $f_1$ and $f_2$ are as in the notation.  Then in view of \eqref{e1}, we have
\begin{equation}\label{eee}
a_0=b_0c_0=\pm p^k d;~a_m=b_{m}c_{n-m}.
\end{equation}
Since each zero $\theta$ of $f$ satisfies $|\theta|>d$, we must have $|b_0/b_m|>d$ and $|c_0/c_{n-m}|>d$ which further give $|b_0|>d$ and $|c_0|>d$.

If $p\nmid c_0$, then $p^k\mid b_0$ and consequently the second equality in \eqref{eee} yields $|c_0|<d$, a contradiction. On the other hand if  $p\mid b_0$ and $p\mid c_0$ then $k>1$ which in view of Lemma \ref{L1} gives the desired contradiction $p\mid a_j$. \qed

\textbf{Proof of Theorem \ref{th:2}.} Suppose to the contrary that $f(x)=f_1(x)f_2(x)$ where $f_1$ and $f_2$ are as in the notation. Then $b_0c_0=a_0$ and $b_mc_{n-m}=a_n=\pm p^k d$. Since each zero $\theta$ of $f$ satisfies $|\theta|>d$, we must have $|b_0/b_m|>d$ and $|c_0/c_{n-m}|>d$.
    If $p\nmid b_m$ then $p^k\mid c_{n-m}$ so that $|b_m|\leq d$ and  we have
    \begin{equation*}
\bigl|{a_0}/{a_n}\bigr|=|b_0/b_m|\times \bigl|{c_0}/{c_{n-m}}\bigr|>|b_0/d|d=|b_0|\geq q,
\end{equation*}
which contradicts the hypothesis.

On the other hand  if $p\mid b_m$ and  $p\mid c_{n-m}$, then $k\geq 2$ which on using Lemma \ref{L1} yield the desired contradiction $p\mid a_{n-j}$. \qed
\begin{remark}
In view of Theorems \ref{th:1}-\ref{th:2}, the hypothesis on zeros of $f$ is not required in the case when $j=n$, wherein the hypothesis on $a_0$ is also not required in Theorem \ref{th:2} and we then have:
\begin{thmx}\label{th:3}
   Let $f=a_0+ a_{1}x+\cdots+a_n x^n\in \Bbb{Z}[x]$ be a primitive polynomial. For a prime $p$ and positive integers $k$ and $n$, if $\gcd(k,n)=1$, $p^k\mid \gcd(a_0,a_1,\ldots,a_{n-1})$, $p\nmid a_{n}$, and $p^{k+1}\nmid a_0$, then $f$ is irreducible in $\Bbb{Z}[x]$.
\end{thmx}
Theorem \ref{th:3} is well known and is generally proved using Newton polygons (see \cite{D}). However here, we provide an alternative proof  based on Lemma \ref{L1}.

\textbf{Proof of Theorem \ref{th:3}.}
To the contrary assume that $f(x)=f_1(x)f_2(x)$ where $f_1$ and $f_2$ are as in the notation. In view of Lemma \ref{L1}, it is enough to show that $p\mid b_0$ and $p\mid c_0$ in order to get the desired contradiction. Since $p\mid a_0=b_0c_0$, we may assume without loss of generality that $p\mid b_0$. Since $p\nmid a_n=b_mc_{n-m}$, we have $p\nmid b_m$ and $p\nmid c_{n-m}$.   So, there exists a least positive integer $t\leq m$ such that $p\nmid b_t$. This in view of \eqref{ee} yields the following:
    \begin{equation*}
   p\mid (a_t-b_0c_t-b_1c_{t-1}-\cdots-b_{t-1}c_1)=b_tc_0,
\end{equation*}
so that $p\mid b_t c_0$, which further gives $p\mid c_0$. \qed
\end{remark}
\section{Examples}
 \textbf{1.} For a prime $p$, positive integers $n$ and $k$ with $\gcd(k,j)=1$, consider the polynomial
\begin{equation}\label{ex1}
X_{j,k}= p^{k+1}(1+x+x^2+\cdots+x^{j-1})+(p^k-1) x^{j}+p^{k-1}x^{j+1}(1+x+\cdots+x^{n-j-1}).
\end{equation}
We will show that each zero $\zeta$ of $X_{j,k}$ satisfies $|\zeta|>1$. Observe that
\begin{equation}\label{ex11}
(x-1)X_{j,k}=-p^{k+1}+(p^{k+1}-p^k+1)x^j+(p^{k}-p^{k-1}-1)x^{j+1}+p^{k-1}x^{n+1}.
\end{equation}
so that the coefficients of $x^j$, $x^{j+1}$, and $x^{n+1}$ in $(x-1)X_{j,k}$ are all positive. If $|\zeta|<1$
then from \eqref{ex11} we have
\begin{equation}\label{ex12}
\begin{split}
p^{k+1}&=(p^{k+1}-p^k+1)\zeta^j+(p^{k}-p^{k-1}-1)\zeta^{j+1}+p^{k-1}\zeta^{n+1}\\
&\leq (p^{k+1}-p^k+1)|\zeta|^j+(p^{k}-p^{k-1}-1)|\zeta|^{j+1}+p^{k-1}|\zeta|^{n+1}\\
&<(p^{k+1}-p^k+1)+(p^{k}-p^{k-1}-1)+p^{k-1}=p^{k+1},
\end{split}
\end{equation}
which is absurd. So we must have $|\zeta|\geq 1$.

If $|\zeta|=1$ for some zero $\zeta$ of $X_{j,k}$, then $\zeta=e^{\iota t}$ for some real number $t$. Now from \eqref{ex12},
$(p^{k+1}-p^k+1)(1-e^{jt})+(p^{k}-p^{k-1}-1)(1-e^{(j+1)t})+p^{k-1}(1-e^{(n+1)t})=0,$ which on comparing real parts gives
$$(p^{k+1}-p^k+1)\sin^2\{(jt/2)\}+(p^{k}-p^{k-1}-1)\sin^2\{(j+1)t/2\}+p^{k-1}\sin^2\{(n+1)t/2\}=0$$ which is possible only if $jt,(j+1)t,(n+1)t\in 2\pi\Bbb{Z}$. Thus we have $\zeta^j=\zeta^{j+1}=\zeta^{n+1}=1$, which give $\zeta=1$. But from \eqref{ex1}, $X_{j,k}(1)>0$ which again leads to a contradiction.
We conclude that each zero $\zeta$ of $X_{j,k}$ satisfies $|\zeta|>1$.

Clearly $X_{j,k}$ satisfies rest of the hypotheses of Theorem \ref{th:1}.
So $X_{j,k}$ is irreducible in $\Bbb{Z}[x]$.

 \textbf{2.} For a prime $p$, positive integers $k$, $n$, $m<p$, and $j\leq n$ with $\gcd(k,j)=1$, the polynomial
\begin{equation*}
Y_{j,k,m}= p^{k}(n+x+x^2+\cdots+x^{n-j-1})+m x^{n-j}+p^kx^{n-j+1}(1+\cdots+x^{j-1})
\end{equation*}
satisfies the hypotheses of Theorem \ref{th:2}. So $Y_{j,k,m}$ is irreducible in $\Bbb{Z}[x]$.

 \textbf{3.} Let $d$ be a positive integer and $f= a_0+a_1x+\cdots+a_n x^n\in \Bbb{Z}[x]$ such that
\begin{equation*}
 |a_0|>|a_1|d+|a_2|d^2+\cdots+|a_n|d^n.
\end{equation*}
Then for $|x|\leq d$, we have
\begin{eqnarray*}
 |f(x)|&\geq & |a_0|-|a_1||x|-\cdots-|a_n| |x|^n> |a_0|-|a_1|d-|a_2|d^2-\cdots-|a_n| d^n>0,
\end{eqnarray*}
which shows that each zero $\theta$ of $f$ satisfies $|\theta|>d$. Now imposing the conditions of Theorem \ref{th:1} or Theorem \ref{th:2} on $f$,  the irreducibility of $f$ in $\Bbb{Z}[x]$ is immediate.

\end{document}